\documentclass[11pt]{amsart}
\usepackage{amsmath, amssymb, eufrak}
\usepackage{graphicx, subfig}
\usepackage{a4wide}
\usepackage{pifont}

\newcommand{\N}{\ensuremath{\mathbb{N}}}
\newcommand{\Z}{\ensuremath{\mathbb{Z}}}

\newcommand{\Q}{\ensuremath{\mathbb{Q}}}

\renewcommand{\S}{\ensuremath{{\mathcal S}}}
\newcommand{\M}{\ensuremath{{\mathcal M}}}
\newcommand{\Pp}{\ensuremath{{\mathcal P}}}
\newcommand{\Cc}{\ensuremath{{\mathcal C}}}
\newcommand{\Rr}{\ensuremath{{\mathcal R}}}

\newcommand{\Gal}{\mbox{Gal}}
\newcommand{\lcm}{\mbox{lcm}}
\renewcommand{\rho}{\varrho}
\renewcommand{\epsilon}{\varepsilon}

\newtheorem{thm}{Theorem}[section]
\newtheorem{lemma}[thm]{Lemma}

\newtheorem{cor}[thm]{Corollary}
\newtheorem{defn}[thm]{Definition}

\newtheorem{rem}[thm]{Remark}

\begin{document}

\title[Bravais colourings of $\M_n$ with CLS1]{On Bravais colourings of cyclotomic integers with class number one}

\author[E.P. Bugarin]{Enrico Paolo Bugarin}
\address{Fakult\"at f\"ur Mathematik, Universit\"at Bielefeld, 33501 Bielefeld, Germany}
\urladdr{http://www.math.uni-bielefeld.de/~pbugarin}
\email{pbugarin@math.uni-bielefeld.de}

\begin{abstract}
Given a Bravais colouring of planar modules $\M_n:=\Z[\xi_n]$, where $\xi_n$ is a primitive $n$th root of unity, 
two important colour groups arise: the colour symmetry group $H$, which permutes the colours of a given colouring of $\M_n$, and the colour preserving group $K$, a normal subgroup of $H$ that fixes the colours. This paper gives a complete characterisation of $H$ and $K$ for all $\ell$-colourings of $\M_n$ for values of $n$ for which $\M_n$ has class number one.
\end{abstract}

\maketitle


\section{\bf Preliminaries} \label{intro}

Consider the planar module $\M_n:=\Z[\xi_n]$, where $\xi_n$ is a primitive $n$th root of unity. 
In particular, $\xi_n=e^{2\pi i/n}$ is taken for convenience. $\M_n=\Z[\xi_n]$ is the ring of cyclotomic integers. For the following values of $n$, $\M_n$ has class number one and so is a principal ideal domain \cite{MM}:
\begin{equation} \label{eq:cls1} 
\begin{array}{rcr}n & = & 3,4,5,7,8,9,11,12,13,15,16,17,19,20,21,24,\\ & & 25,27,28,32,33,35,36,40,44,45,48,60,84. \end{array}
\end{equation}
Note that $M_1=M_2=\Z$ is $1$-dimensional. In general, since $\M_n=\M_{2n}$ whenever $n$ is odd, values of $n \equiv 2\mod 4$ do not appear in the list above for the sake of uniqueness. Throughout the text, $n$ only refers to the values listed in Eq. \ref{eq:cls1}. These values of $n$ are naturally grouped according to the Euler's totient function $\phi(n)$, where $\phi(n)=|\{1\leq k\leq n: \gcd(k,n)=1\}|$. The case $\phi(n) = 2$ covers the two crystallographic cases \cite{mlp}, while $\phi(n) = 4$ covers the standard symmetries of genuine planar quasicrystals \cite{bugs}.

\begin{defn}
Let $(q)$ denote the ideal generated by $q \in \M_n$. An ideal colouring $c$ of $\M_n$ with $\ell$ colours occurs as follows: Let $q\in\M_n$ such that $[\M_n :(q)]=\ell$. 
For each $z\in(q)=q\M_n$, let $c(z)=1$. Let the other cosets of $(q)$ be $(q)+t_2,\ldots,(q)+t_\ell$. For each $z\in (q)+t_i$, let $c(z)=i$. 
\end{defn}

If $(q)$ is given explicitly, such an ideal colouring may also be referred to as a colouring induced by $(q)$ or $q$, or furthermore, an $\ell$-colouring induced by $(q)$ given that $[\M_n :(q)]=\ell$. Let us denote the algebraic norm of $(q)$ in $\M_n$ by $N_n(q)$, and recall that $N_n(q)=\prod_{i}\sigma_i(q)$, where $\sigma_i\in\Gal(\Q(\xi_n):\Q)$. Since $\Gal(\Q(\xi_n):\Q)\simeq(\Z/n\Z)^\times$, then $|\Gal(\Q(\xi):\Q)|=\phi(n)$. Equivalently, the algebraic norm of $(q)$ equals its index in $\M_n$, and so $[\M_n :(q)]=N_n(q)$. Since $\M_n$ is a principal ideal domain for values of $n$ in Eq. \ref{eq:cls1}, every non-zero $q\in\M_n$ thus induces an ideal colouring as defined above. Conversely, for every ideal colouring in $\M_n$ there exists a $q\in\M_n$ that induces it. Also, as shown in \cite{jj}, an ideal $\ell$-colouring of $\M_n$ is equivalent to a Bravais $\ell$-colouring of $\M_n$ for values of $n$ listed in Eq. \ref{eq:cls1}. Thus from now on, an ideal colouring of $\M_n$ is formally referred to as a Bravais colouring of $\M_n$. For an actual definition of a Bravais colouring of the module $\M_n$, see \cite{bg}.

In analysing a Bravais colouring $c$ of $\M_n$, we consider the {\em symmetry group} $G$ of the uncoloured module $\M_n$. This group is {\em symmorphic}, that is, $G = \M_n \rtimes D_N$, the semi-direct product of its translation group (being naturally isomorphic to $\M_n$) with its point group, the dihedral group $D_N$, where $N=\lcm(2,n)$. For a given colouring $c$, we also consider the following subgroup of $G$, the group $$H := \{ h \in G \, | \, \exists \, \pi \in \S_{\ell} \;  \forall x \in \M_n: \; c(h(x)) = \pi (c(x))\},$$ where $\S_{\ell}$ denotes the symmetric group on $\ell$ letters and $c(x)$ denotes the colour of $x$. The elements of $H$ are called the {\em colour symmetries} of $\M_n$ and $H$ is the {\em colour symmetry group} of the corresponding colouring $c$ of $\M_n$.

By the requirement $ch=\pi c$, each $h \in H$ determines a unique permutation $\pi = \pi_h$. This also defines a map $$P: H \to \S_{\ell}, \quad P(h):=\pi_h.$$ Let $g,h\in H$. Because of $c (hg(x)) = ch(g(x)) = \pi_h c(g(x)) = \pi_h (\pi_g (c(x))) = \pi_h \pi_g (c(x))$, $P$ is a group homomorphism.

Another group of interest is the subgroup of $H$ which consists of elements that fix the colours of a colouring $c$ of $\M_n$, called the {\em colour preserving group}, 
$$K := \{ k \in H \, | \, \forall x \in \M_n:\: c(k(x)) = c(x)\}.$$ In other words, $K$ is the kernel of $P$, a normal subgroup of $H$. This paper gives a complete characterisation of the colour groups $H$ and $K$ for the Bravais colourings of $\M_n$.

\section{\bf The characterisation of {\boldmath $H$}}

\begin{defn}
A Bravais colouring $c$ of $\M_n$ is called perfect if its colour symmetry group $H$ equals $G$. It is called chirally perfect if $H = G'$, where $G'$ is the index $2$ subgroup of $G$ consisting of the orientation preserving isometries in $G$, that is, $G'=\M_n\rtimes C_N$, where $C_N$ is the cyclic group of order $N$, where $N=\lcm(2,n)$.
\end{defn}

As shown in \cite{jj}, a Bravais colouring of $\M_n$ is only either perfect or chirally perfect, depending on the factorisation of $(q)$, the ideal that induces the colouring. Recall that $\M_n$ is a principal ideal domain and so is a unique factorisation domain. 

The unique factorisation of $q$ over $\M_n$ with class number one reads
$$q = \varepsilon \prod_{p_i \in \Pp} p_i^{\alpha_i} \prod_{p_j \in \Cc}
 \omega_{p_j}^{\beta_j}  \overline{\omega}_{p_j}^{\gamma_j}
 \prod_{p_k \in \Rr} p_k^{\delta_k},$$
where $\varepsilon$ is a unit in $\M_n$ and $\omega_{p_j} \overline{\omega}_{p_j} = p_j$.
Here, $\Pp$, $\Cc$, and $\Rr$ respectively denote the set of inert, complex splitting, and ramified primes over $\M_n$. The generator
$q$ is called {\em balanced} if $\beta_j=\gamma_j$ for all $j$, meaning to say that $q$ is balanced if it is of the form $$q= \epsilon x p,$$
where $\epsilon$ is a unit in $\M_n$, $x$ is a real number in $\M_n$ ($x \in \Z[\xi_n+\overline{\xi}_n]$),
and $p$ is a product of ramified primes. By the definition of a ramified prime $p$ (see \cite{wash}), $\overline{p} \in (p)$ holds in $\M_n$. (Equivalently, $p / \overline{p}$ is a unit in $\M_n$.) Recall that all units $\epsilon$ in $\Z[\xi_n]$ are of the form $\epsilon=\pm \lambda \xi_n^k$, where $\lambda \in \Z[\xi_n+\overline{\xi}_n]$, compare \cite{jj}.

\smallskip

In \cite{jj}, it is shown that a Bravais colouring of $\M_n$ induced by $(q)$ is perfect ($H=G$) if and only if $q$ is balanced if and only if $(q)=(\overline{q})$. Otherwise, $H=G'$.

For $p>2$ prime in $\Z$, $p$ ramifies completely in $\M_p$, that is, $(p)=\prod_{i=1}^{p-1}(1-\xi_p^i)=(1-\xi_p)^{p-1}$. In general, for any prime $p$, such that $n=rp^t$ with $p\nmid r$, $(p)=(1-\xi_{p^t})^{\phi(n)/\phi(r)}$. Hence, the only factors of ramified primes in $\M_n$ are those of the form $(1-\xi_{p^t})^j$, where $p^t$ is the largest prime power of $p$ that divides $n$. Note that $(1-\xi_{p^t})$ is balanced. In the case when $r\neq1$, $(1-\xi_{p^t})$ may split further, say $(1-\xi_{p^t})=(q_1)(q_2)\cdots(q_m)$. (See \cite{wash}.) However, not a single $(q_i)$ here is balanced. The first example is $n=20$ for $p=5$. Here, $(1-\xi_5)$ splits as $(1+\xi_{20}-\xi_{20}^3)(\overline{1+\xi_{20}-\xi_{20}^3})$. Neither of the factors of $(1-\xi_5)$ in $\M_{20}$ is balanced.

\begin{lemma}
If $q$ is balanced, then each $\sigma_i(q)$ is also balanced, where $\sigma_i\in\Gal(\Q(\xi_n):\Q)$ for $1\leq i\leq\phi(n)$; more so, $q$ and $\sigma_i(q)$ induce the same colouring of $\M_n$.
\end{lemma}

\begin{proof}
Since $q$ is balanced, then it is of the form $q=\epsilon xp$, where $\epsilon$ is a unit in $\M_n$, $x$ is a real number in $\M_n$, and $p$ is a product of ramified primes. Thus, $\sigma_i(q)=\sigma_i(\epsilon xp)=\epsilon'x\sigma_i(p)$, where $\sigma_i(\epsilon)=\epsilon'$ is still a unit in $\M_n$ and $(\sigma_i(p))=(p)$, making $\sigma_i(q)$ balanced as well. More so, $(\sigma_i(q))=(q)$, thus inducing the same colouring.
\end{proof}

\begin{cor} \label{unique}
If $q_1$ and $q_2$ induce a perfect $\ell$-colouring of $\M_n$, then $(q_1)=(q_2)$. That is, if there exists a perfect $\ell$-colouring of $\M_n$, then that colouring is unique (up to permutation of colours) and there can't be two distinct perfect $\ell$-colourings of $\M_n$.
\end{cor}

Recall that the converse of Corollary \ref{unique} is true in general, that is, given a unique $\ell$-colouring of $\M_n$, say as induced by $(q)$, then that colouring must be perfect, or else $(q)$ and $(\overline{q})$ would induce two distinct $\ell$-colourings.


\begin{lemma} \label{G'}
Consider an $\ell$-colouring of $\M_n$ induced by $(q)$, such that $(q)\neq(d)$ for any $d\in\Z$. If $\ell>1$ and $\gcd(\ell,n)=1$, then $H=G'$. 
\end{lemma}

\begin{proof}
The only factors of ramified primes in $\M_n$ are those of the form $(1-\xi_{p^t})$, where $p^t$ is the largest prime power of $p$ that divides $n$. Note that $N_n(1-\xi_{p^t})=p^{\phi(r)}$, when $n=rp^t$, compare \cite{philmag}. Thus, whenever $q$ is balanced such that $N_n(q)=\ell>1$, then $\ell$ must be a multiple of some prime $p$, which also divides $n$. Otherwise, $q$ is not balanced and so $H=G'$.
\end{proof}

Knowing what factors yield perfect colourings of $\M_n$, one can list the conditions for $q$ being balanced. This can be done in two ways, either via the factorisation of $q$ or via the actual values of $\ell$, the algebraic norm of $(q)$. We give both characterisations. We begin with the factorisation of $(q)$, which is given as follows.

\begin{thm}\label{thmH}
A perfect Bravais colouring of $\M_n$ ($H=G$) induced by $(q)$ exists if and only if $(q)$ equals one of the following forms below. Otherwise, the colouring is chirally perfect ($H=G'$).
\begin{enumerate}
\item[i.] $(q)=(d)$, for some $d\in\Z$.
\item[ii.] $(q)=(1-\xi_{p^t})^s,$ $s>0$ for some prime $p\in\Z$ such that $p^t$ is the largest prime power of $p$ that divides $n$.
\item[iii.] Any product of the forms given in i. and ii.
\end{enumerate}
\end{thm}

\begin{proof}
The {\em if} part of the theorem is straightforward, since $d$ and $1-\xi_{p^t}$ are balanced, and so is any poduct of these factors. Conversely, if $(q)$ induces a perfect colouring of $\M_n$, then $q$ has to be of the form $q=\epsilon xp$, where $\epsilon$ is a unit in $\M_n$, $x$ is a real number in $\M_n$, and $p$ is a product of factors of ramified primes. The case $(p)=(1)$ implies i., the case $(x)=(1)$ implies ii.; otherwise we get case iii.
\end{proof}

Equivalently to the previous theorem, one can formulate the characterisation of $H$, based on the value of $\ell$, the algebraic norm of $(q)$.

\begin{thm}\label{thmHL}
A perfect Bravais $\ell$-colouring of $\M_n$ ($H=G$) exists if and only if $\ell$ takes one of the following values. Otherwise, the colouring is chirally perfect ($H=G'$). In some cases where a perfect $\ell$-colouring exists and yet there are still other $\ell$-colourings, then the other $\ell$-colourings must be chirally perfect.
\begin{enumerate}
\item[i.] $\ell=d^{\phi(n)}$ for some $d\in\Z$.
\item[ii.] $\ell=p^{\phi(r)s}$ for $s>0$, where $p$ prime in $\Z$ divides $n$ and $r$ is the $p$-free part of $n$, that is, $n=rp^t$ for some $t$ such that $p$ does not divide $r$.
\item[iii.] Any product of the values given in i. and ii.
\end{enumerate}
\end{thm}

\begin{proof}
We prove the {\em if} part of the the theorem as follows: First, if $\ell=d^{\phi(n)}$ for some $d\in\Z$, then taking $q=d$ makes $q$ balanced and so induces a perfect colouring of $\M_n$. Second, if $\ell=p^{\phi(r)s}$, where $r$ and $s$ are given as above, we can find a perfect colouring as induced by $q=(1-\xi_{p^t})^s$, since $N_n((1-\xi_{p^t})^s)=p^{\phi(r)s}$. Finally, iii. follows from the multiplicative property of $N_n(q)=\ell$. 

The {\em only if} part is similar to the proof of Theorem \ref{thmH}. Note that for certain cases, where there are more than one $\ell$-colourings of $\M_n$, with one being perfect, Corollary \ref{unique} guarantees that the other colourings are only chirally perfect.
\end{proof}

Thus, for a planar module $\M_n$, the values of $\ell$ for which a perfect colouring exists are those obtained from $\ell=d_{\phantom{1}}^{\phi(n)s_0}p_1^{\phi(r_1)s_1}\cdots p_k^{\phi(r_k)s_k}$, where $d\in\Z^+$, $p_i$'s are the primes that divide $n$, and $r_i$ is the $p_i$-free part of $n$, for any $s_i \geq 0$. For example, take $n=24$ and consider $\M_{24}$. Two primes divide $24$, namely $2$ and $3$. The respective $2$- and $3$-free parts of $24$ are $3$ and $8$. Since $\phi(3)=2$, $\phi(8)=4$, and $\phi(24)=8$, we get $\ell = d^{8s_0}\cdot2^{2s_1}\cdot3^{4s_2}=d^{8s_0}\cdot4^{s_1}\cdot81^{s_2}$. Plugging in values for $d,s_0,s_1$, and $s_2$ yields the indices for which a perfect Bravais colouring of $\M_{24}$ exists. Furthermore, together with Theorems \ref{thmH} and \ref{thmHL}, given a perfect $\ell$-colouring, one can find $(q)$ having $\ell$ as its index in $\M_n$.

\begin{rem}When a non-perfect colouring exists though, it is not very easy to find $q$, since factorisation in $\M_n$ in general is not straightforward especially with splitting primes. For explicit examples on non-perfect (i.e. chirally perfect) colourings of $\M_n$, see \cite{bugs, jj, philmag}. \end{rem}

Let $a_n(\ell)$ be the number of distinct (up to permutation of colours) Bravais $\ell$-colourings of $\M_n$. Baake and Grimm in \cite{bg} gave a number theoretic formulation that determines $a_n(\ell)$ for all $\ell>0$. This is done by formulating a Dirichlet series generating function $F_n(s)$ as the Dedekind zeta function of the cyclotomic field $\Q(\xi_n)$, namely: $$F_n(s):=\sum_{\ell=1}^{\infty}\frac{a_n(\ell)}{\ell^s}=\zeta_{\Q(\xi_n)}(s):=\sum_{q}\frac{1}{N_n(q)^s},\quad\mbox{Re}(s)>1,$$ where $q$ runs over all non-zero element of $\M_n=\Z[\xi_n]$. Further they show that $$F_n(s):=\sum_{\ell=1}^{\infty}\frac{a_n(\ell)}{\ell^s}=\prod_{p}E_n(p^{-s}),$$ where $p$ runs over the primes in $\Z$, and that $$E_n(p^{-s})=\sum_{j=0}^{\infty}{{j+m-1}\choose{m-1}}\frac{1}{(p^s)^{lj}}.$$ Note that for primes not dividing $n$, $p \equiv k \mod n$ contribute via $p^l$ as basic index, where $l$ is the smallest integer such that $k^l \equiv 1\mod n$, and $m=\phi(n)/l$. For primes dividing $n$, one computes the $p$-free part of $n$, say $r$, and replace $n$ with $r$ to determine $l$ and $m$ as in the case where $p$ does not divide $n$. For general primes $p$ and those that ramify in $\M_n$, values for $k,m,l$ are already given in \cite{bg}.

Relating Theorem \ref{thmHL} to the Dirichlet series given in \cite{bg}, one can easily generate a series which gives the indices of perfect colourings of $\M_n$. This can be done by slightly modifying the formulation of $E_n(p^{-s})$ by removing all indices that come from primes that split. Hence we have the following modified formulation of $E_n(p^{-s})$ that only generates basic indices for which a perfect Bravais $\ell$-colouring exists.$$E_n^G(p^{-s})=\sum_{j\in m\N^*}{{j+m-1}\choose{m-1}}\frac{1}{(p^s)^{lj}},\quad\mbox{where }m\N^*=\{mk|k\in\N\cup\{0\}\}.$$

Table \ref{tabH} lists the first few terms of the modified Dirichlet series which generates terms that give perfect $\ell$-colourings of $\M_n$. Note that for certain cases in Table \ref{tabH} where $a_n(\ell)>1$, only one perfect $\ell$-colouring exists and the others are chirally perfect.

It is possible to further simplify $E_n^G(p^{-s})$ by making all the numerators equal to 1, since we know that if there exists a perfect $\ell$-colouring, then it has to be unique. However, by doing this, we also eliminate noting when the existence of an $\ell$-colouring of $\M_n$ is shared by both perfect and chirally perfect colourings.

\section{\bf The characterisation of {\boldmath $K$}}

In \cite{philmag}, divisibility conditions for $\ell$ are derived. These conditions determine when an index $\ell$ may yield a non-trivial colour preserving group $K$. We say that $K$ is trivial when $K=T_{(q)}$, where $T_{(q)}$ is the group of translations generated be elements of $(q)$. Naturally, $T_{(q)}\simeq q\M_n=(q)$. For any Bravais colouring of $\M_n$ induced by $(q)$, $T_{(q)}\leq K$, and in general, $K$ is of the form $K=T_{(q)}\rtimes P_n$, where $P_n$ is a subgroup of $D_n$. Recall that we don't need to consider $N=\lcm(2,n)$ here anymore, because whenever $n$ is odd, a $2n$-fold rotation about the origin can no way be in $K$ (unless $(q)=(1)=\M_n$). See \cite{bugs, jj, philmag} for details.

\begin{lemma}
If $\ell$ is divisible by two distinct primes, $K=T_{(q)}$.
\end{lemma}

\begin{proof}
If $K$ is non-trivial, then $\ell$ must divide some prime power, say $p^s$, see \cite{philmag}. However, if $\ell$ is divisible by two distinct primes, $\ell$ cannot divide any prime power $p^s$. Therefore, $K=T_{(q)}$.
\end{proof}

The following lemmas determine which subgroup of $D_n$ is in $K$. 
Considering $\M_n$ as a subset of the complex plane (which is dense for $n>4$), let $R_k$ be the $k$-fold rotation about the origin, and let $S$ be the reflection along the real line. It then follows that $R_k(x)=\xi_n^{n/k}x$ and $S(x) = \overline{x}$ for $x\in\M_n$. Note that given a Bravais colouring induced by $(q)$, $R_k(q)\in(q)$ and when $q$ is balanced, $S(q)\in(q)$.

\begin{lemma}\label{lem}Consider a Bravais colouring $c$ of $\M_n$ induced by $q$. Then $\langle R_k\rangle \leq K$ if and only if $(1-\xi_n^{n/k}) \subseteq (q)$. Similarly, $\langle S\rangle \leq K$ if and only if $(1-\xi_n^{n-2i}) \subseteq (q)$ for $i=1,2,\ldots,n$.
\end{lemma}

\begin{proof}
Suppose $R_k\in K$, then $c(1)=c(\xi_n^{n/k})$, and so $(1-\xi_n^{n/k}) \subseteq (q)$. Similarly, if $S\in K$, then $c(\xi_n^i)=c(\overline{\xi}_n^i)=c(\xi_n^{n-i})$ implying that $(1-\xi_n^{n-2i}) \subseteq (q)$ for any $i=1,2,\ldots,n$. Conversely, consider any non-trivial coset of $(q)$ in $\M_n$, say $(q)+s$, where $s\in\M_n$ and $s\notin(q)$. The rotation $R_k$ then maps the coset $(q)+s$ to $(q)+\xi_n^{n/k}s$. These two cosets are equal if and only if $(1-\xi_n^{n/k}) \subseteq (q)$, implying that $R_k$ is in $K$. A similar argument holds for the case of $S$.
\end{proof}

As the previous lemma suggests, checking for $S$ requires checking $1-\xi_n^{j}$ for values of $j=n-2i$, $i=1,2,\ldots,n$. That is, if $n$ is even, one only checks the even powers of $\xi_n$, but when $n$ is odd, all powers of $\xi_n$ have to be checked. This brings us to the next result.

\begin{cor}
Suppose $n$ is divisible by two distinct primes, then $S\in K$ if and only if $\ell=1$.\end{cor}

\begin{proof}
Recall that if $(q')\subseteq(q)$, then $\ell=N_n(q)$ divides $N_n(q')$. If $n$ is divisible by two distinct primes then $N_n(1-\xi_n^2)=1$, see \cite{philmag}. Then $S\in K$ if and only if $\ell$ divides 1.
\end{proof}



\begin{lemma}
For odd values of $n$: If $\ell=2^{\phi(n)/2}$, then $R_2\in K$; more so, $K=T_{(q)}\rtimes C_2$.
\end{lemma}

\begin{proof}
Similar to the proof of the previous lemma, one sees that $R_2\in K$ if and only if $(2) \subseteq (q)$. In particular, if $R_2\in K$, then $\ell$ divides $2^{\phi(n)}$. The case $\ell=2^{\phi(n)/2}$ only occurs when $2$ splits in $\M_n$. This happens for the following values of $n$: 7, 15, 17, 21, 33, 35, 45 (essentially, Thm 2.23 of \cite{wash}, also see \cite{bg}). More so, since $\ell=2^{\phi(n)/2}$, then it cannot divide other prime powers aside from $2^{\phi(n)}$, hence no other subgroup of $D_n$ is in $K$ aside from $\langle R_2\rangle$, and so $K=T_{(q)}\rtimes C_2$, compare \cite{philmag}.
\end{proof}

The previous lemmas enable us to determine all non-trivial cases of $K$ for all the values of $n$ given in Eq. \ref{eq:cls1}. The list is given in Table \ref{tabK}, and formally we write this as Theorem \ref{thmK} below. For the three remarks given in this theorem, Thm 2.23 of \cite{wash} is used.

\begin{thm}\label{thmK}
Consider an $\ell$-colouring of $\M_n$ for values of $n$ listed in Eq. \ref{eq:cls1}. The colour preserving group $K$ is trivial, that is $K=T_{(q)}$, except for those values of $\ell$ listed in Table \ref{tabK}. In this table, $\ell\lrcorner P_n$ denotes that an $\ell$-colouring of $\M_n$ yields $K=T_{(q)}\rtimes P_n$. For certain cases with an asterisk $(^*)$, we have the following:
\begin{enumerate}
\item[(i)] $(2)$ splits into two distinct primes in $\M_n$ for $n=7,15,17,21,33,35,45:$ In these cases let $(2)=(q_2)(\overline{q}_2)$. The colourings induced by $2, q_2,$ or $\overline{q}_2$ yield $P_n=C_2$. However, $K$ is trivial for the colourings induced by $q_2^2$ or $\overline{q}_2^2$.
\item[(ii)] For primes $p=3,5,7$, $(1-\xi_p)$ may split depending on $n$. When it does, it splits into two distinct primes in $\M_n$. In any case, we let $(1-\xi_p)=(q_p)(\overline{q}_p)$. Whenever $(1-\xi_p)$ splits, the colourings induced by $1-\xi_p, q_p,$ or $\overline{q}_p$ yield $P_n=C_p$, while the colourings induced by $q_p^2$ or $\overline{q}_p^2$ yield a trivial $K$. For $n=48,60,84$, $(1-\xi_3)$ splits; for $n=40,60$, $(1-\xi_5)$ splits; and for $n=21,84$, $(1-\xi_7)$ splits; all in the way just described.
\item[(iii)] Finally, $(1-\xi_4)$ splits into two distinct primes in $\M_n$ for $n=28,60,84$. In these cases let $(1-\xi_4)=(q_4)(\overline{q}_4)$. Write $(2)=(1-\xi_4)(1-\xi_4)=(q_4)(\overline{q}_4)(q_4)(\overline{q}_4)$. Consider the expression $q_4^i(1-\xi_4)^j$ or $\overline{q}_4^i(1-\xi_4)^j$. For the following triples $(i,j,\ell)$, $q_4^i(1-\xi_4)^j$ or $\overline{q}_4^i(1-\xi_4)^j$ induces an $\ell$-colouring: $(0,1,2^{\phi(n)/2}),(1,0,2^{\phi(n)/4})$ with $P_n=C_4;$ $(0,2,2^{\phi(n)}),$ $(2,0,2^{\phi(n)/2}),$ $(1,1,2^{3\phi(n)/4})$ with $P_n=C_2;$  and $(4,0,2^{\phi(n)}),$ $(2,1,2^{\phi(n)}),$ $(3,0,2^{3\phi(n)/4})$ with trivial $K$.
\end{enumerate}
\end{thm}

Theorem \ref{thmK} (iii) explains an interesting case for $n=28,60,84$. In these cases, there are two $2^{\phi(n)/4}$-colourings, three $2^{\phi(n)/2}$-colourings, four $2^{3\phi(n)/4}$-colourings, and five $2^{\phi(n)}$-colourings of $\M_n$. Assuming $(1-\xi_4)=(q_4)(\overline{q}_4)$, then the two $2^{\phi(n)/4}$-colourings are induced by $q_4$ and $\overline{q}_4$; the three $2^{\phi(n)/2}$-colourings are induced by $1-\xi_4$, $q_4^2$, and $\overline{q}_4^2$; the four $2^{3\phi(n)/4}$-colourings are induced by $q_4(1-\xi_4)$, $\overline{q}_4(1-\xi_4)$, $q_4^3$, and $\overline{q}_4^3$; and the five $2^{\phi(n)}$-colourings are induced by $2$, $q_4^2(1-\xi_4)$, $\overline{q}_4^2(1-\xi_4)$, $q_4^4$, and $\overline{q}_4^4$.





\section*{Acknowledgement}

It is a pleasure to thank Michael Baake for helpful discussion. This work is supported by DFG via the Collaborative Research Centre 701 through the faculty of Mathematics at the University of Bielefeld.

\begin{table}
{\small
\caption{\label{tabH} First terms of the modified Dirichlet series $E_n^G(p^{-s})$ that generates indices that give perfect $\ell$-colourings of $\M_n$. For each of these indices, there exists a unique perfect colouring.}
\begin{tabular}{cp{400pt}} \hline $n$ & $\zeta'_{\M_n}(s)$ \\[3pt] \hline \\[-10pt]
$3$ & $1 + \frac{1}{3^s} + \frac{1}{4^s} + \frac{1}{9^s} + \frac{1}{12^s} + \frac{1}{16^s} + \frac{1}{25^s} + \frac{1}{27^s} + \frac{1}{36^s} + \frac{1}{48^s} + \frac{3}{49^s} + \frac{1}{64^s} + \frac{1}{75^s} + \frac{1}{81^s} + \frac{1}{100^s}+\frac{1}{108^s}  + \cdots$
\\[3pt] $4$ & $1+\frac{1}{2^s} + \frac{1}{4^s} + \frac{1}{8^s} + \frac{1}{9^s} + \frac{1}{16^s} + \frac{1}{18^s} + \frac{3}{25^s} + \frac{1}{32^s} + \frac{1}{36^s} + \frac{1}{49^s} + \frac{3}{50^s} + \frac{1}{64^s} + \frac{1}{72^s} + \frac{1}{98^s} + \frac{3}{100^s} + \cdots$
\\[3pt] $5$ & $1+\frac{1}{5^s}+ \frac{1}{16^s}+ \frac{1}{25^s}+ \frac{1}{80^s}+ \frac{1}{81^s}+\frac{1}{125^s}+\frac{1}{256^s}+\frac{1}{400^s}+\frac{1}{405^s}+\frac{1}{625^s}+\frac{1}{1280^s}+\frac{1}{1296^s}+\frac{1}{2000^s}+\cdots$
\\[3pt] $7$ & $1+\frac{1}{7^s}+\frac{1}{49^s}+\frac{3}{64^s}+\frac{1}{343^s}+\frac{3}{448^s}+\frac{1}{729^s}+\frac{1}{2401^s}+\frac{3}{3136^s}+\frac{5}{4096^s}+\frac{1}{5103^s}+\frac{1}{15625^s}+\frac{1}{16807^s}+\cdots$
\\[3pt] $8$ & $1+\frac{1}{2^s}+\frac{1}{4^s}+\frac{1}{16^s}+\frac{1}{32^s}+\frac{1}{64^s}+\frac{3}{81^s}+\frac{1}{128^s}+\frac{3}{162^s}+\frac{1}{256^s}+\frac{3}{324^s}+\frac{1}{512^s}+\frac{3}{625^s}+\frac{1}{1024^s}+\frac{3}{1250^s}+\cdots$
\\[3pt] $9$ & $1+\frac{1}{3^s}+\frac{1}{9^s}+\frac{1}{27^s}+\frac{1}{64^s}+\frac{1}{81^s}+\frac{1}{192^s}+\frac{1}{243^s}+\frac{1}{576^s}+\frac{1}{729^s}+\frac{1}{1728^s}+\frac{1}{2187^s}+\frac{1}{4096^s}+\frac{1}{5184^s}+\frac{1}{6561^s}+\cdots$
\\[3pt] $11$ & $1+\frac{1}{11^s}+\frac{1}{121^s}+\frac{1}{1024^s}+\frac{1}{1331^s}+\frac{1}{11264^s}+\frac{1}{14641^s}+\frac{3}{59049^s}+\frac{1}{161051^s}+\frac{3}{649539^s}+\frac{1}{1362944^s}+\cdots$
\\[3pt] $12$ & $1+\frac{1}{4^s}+\frac{1}{9^s}+\frac{1}{16^s}+\frac{1}{36^s}+\frac{1}{64^s}+\frac{1}{81^s}+\frac{1}{144^s}+\frac{1}{256^s}+\frac{1}{324^s}+\frac{1}{576^s}+\frac{3}{625^s}+\frac{1}{729^s}+\frac{1}{1024^s}+\frac{1}{1296^s}+\cdots$
\\[3pt] $13$ & $1+\frac{1}{13^s}+\frac{1}{169^s}+\frac{1}{2197^s}+\frac{1}{4096^s}+\frac{1}{28561^s}+\frac{1}{53248^s}+\frac{1}{371293^s}+\frac{35}{513441^s}+\frac{1}{692224^s}+\frac{1}{4826809^s}+\cdots$
\\[3pt] $15$ & $1+\frac{1}{25^s}+\frac{1}{81^s}+\frac{3}{256^s}+\frac{1}{625^s}+\frac{1}{2025^s}+\frac{5}{4096^s}+\frac{3}{6400^s}+\frac{1}{6561^s}+\frac{1}{15625^s}+\frac{3}{20736^s}+\frac{1}{50626^s}+\frac{7}{65536^s}+\cdots$
\\[3pt] $16$ & $1+\frac{1}{2^s}+\frac{1}{4^s}+\frac{1}{8^s}+\frac{1}{16^s}+\frac{1}{32^s}+\frac{1}{64^s}+\frac{1}{128^s}+\frac{1}{256^s}+\frac{1}{512^s}+\frac{1}{1024^s}+\frac{1}{2048^s}+\frac{1}{4096^s}+\frac{3}{6561^s}+\frac{1}{8192^s}+\cdots$
\\[3pt] $17$ & $1+\frac{1}{17^s}+\frac{1}{289^s}+\frac{1}{4913^s}+\frac{3}{65536^s}+\frac{1}{83521^s}+\frac{3}{1114112^s}+\frac{1}{1419857^s}+\frac{3}{18939904^s}+\frac{1}{24137569^s}+\frac{1}{43046721^s}+\cdots$
\\[3pt] $19$ & $1+\frac{1}{19^s}+\frac{1}{361^s}+\frac{1}{6859^s}+\frac{1}{130321^s}+\frac{1}{262144^s}+\frac{1}{2476099^s}+\frac{1}{4980736^s}+\frac{1}{47045881^s}+\frac{1}{94633984^s}+\cdots$
\\[3pt] $20$ & $1+\frac{1}{16^s}+\frac{3}{25^s}+\frac{1}{256^s}+\frac{3}{400^s}+\frac{5}{625^s}+\frac{1}{4096^s}+\frac{3}{6400^s}+\frac{3}{6561^s}+\frac{5}{10000^s}+\frac{7}{15625^s}+\frac{1}{65536^s}+\frac{3}{102400^s}+\cdots$
\\[3pt] $21$ & $1+\frac{3}{49^s}+\frac{1}{729^s}+\frac{5}{2401^s}+\frac{3}{4096^s}+\frac{3}{35721^s}+\frac{9}{200704^s}+\frac{1}{531441^s}+\frac{5}{1750329^s}+\frac{3}{2985984^s}+\frac{7}{5764801^s}+\cdots$
\\[3pt] $24$ & $1+\frac{1}{4^s}+\frac{1}{16^s}+\frac{1}{64^s}+\frac{3}{81^s}+\frac{1}{256^s}+\frac{3}{324^s}+\frac{1}{1024^s}+\frac{3}{1296^s}+\frac{1}{4096^s}+\frac{3}{5184^s}+\frac{5}{6561^s}+\frac{1}{16384^s}+\cdots$
\\[3pt] $25$ & $1+\frac{1}{5^s}+\frac{1}{25^s}+\frac{1}{125^s}+\frac{1}{625^s}+\frac{1}{3125^s}+\frac{1}{15625^s}+\frac{1}{78125^s}+\frac{1}{390625^s}+\frac{1}{1048576^s}+\frac{1}{1953125^s}+\frac{1}{5242880^s}+\cdots$
\\[3pt] $27$ & $1+\frac{1}{3^s}+\frac{1}{9^s}+\frac{1}{27^s}+\frac{1}{81^s}+\frac{1}{243^s}+\frac{1}{729^s}+\frac{1}{2187^s}+\frac{1}{6561^s}+\frac{1}{19683^s}+\frac{1}{59049^s}+\frac{1}{177147^s}+\frac{1}{262144^s}+\cdots$
\\[3pt] $28$ & $1+\frac{3}{64^s}+\frac{1}{2401^s}+\frac{5}{4096^s}+\frac{3}{153664^s}+\frac{7}{262144^s}+\frac{3}{531441^s}+\frac{1}{5764801^s}+\frac{5}{9834496^s}+\frac{9}{16777216^s}+\frac{9}{34012224^s}+\cdots$
\\[3pt] $32$ & $1+\frac{1}{2^s}+\frac{1}{4^s}+\frac{1}{8^s}+\frac{1}{16^s}+\frac{1}{32^s}+\frac{1}{64^s}+\frac{1}{128^s}+\frac{1}{256^s}+\frac{1}{512^s}+\frac{1}{1024^s}+\frac{1}{2048^s}+\frac{1}{4096^s}+\frac{1}{8192^s}+\cdots$
\\[3pt] $33$ & $1+\frac{1}{121^s}+\frac{1}{14641^s}+\frac{3}{59049^s}+\frac{3}{1048576^s}+\frac{1}{1771561^s}+\frac{3}{7144929^s}+\frac{3}{126877696^s}+\frac{1}{214358881^s}+\frac{3}{864536409^s}+\cdots$
\\[3pt] $35$ & $1+\frac{1}{2401^s}+\frac{1}{15625^s}+\frac{1}{5764801^s}+\frac{3}{16777216^s}+\frac{1}{37515625^s}+\frac{1}{244140625^s}+\frac{1}{13841287201^s}+\frac{3}{40282095616^s}+\cdots$
\\[3pt] $36$ & $1+\frac{1}{9^s}+\frac{1}{64^s}+\frac{1}{81^s}+\frac{1}{576^s}+\frac{1}{729^s}+\frac{1}{4096^s}+\frac{1}{5184^s}+\frac{1}{6561^s}+\frac{1}{36864^s}+\frac{1}{46656^s}+\frac{1}{59049^s}+\frac{1}{262144^s}+\cdots$
\\[3pt] $40$ & $1+\frac{1}{16^s}+\frac{1}{256^s}+\frac{3}{625^s}+\frac{1}{4096^s}+\frac{3}{10000^s}+\frac{1}{65536^s}+\frac{3}{160000^s}+\frac{5}{390625^s}+\frac{1}{1048576^s}+\frac{3}{2560000^s}+\cdots$
\\[3pt] $44$ & $1+\frac{1}{121^s}+\frac{1}{1024^s}+\frac{1}{14641^s}+\frac{1}{123904^s}+\frac{1}{1048576^s}+\frac{1}{1771561^s}+\frac{1}{14992384^s}+\frac{1}{126877696^s}+\frac{1}{214358881^s}+\cdots$
\\[3pt] $45$ & $1+\frac{1}{81^s}+\frac{1}{6561^s}+\frac{1}{15625^s}+\frac{1}{531441^s}+\frac{1}{1265625^s}+\frac{3}{16777216^s}+\frac{1}{43046721^s}+\frac{1}{102515625^s}+\frac{3}{1358954496^s}+\cdots$
\\[3pt] $48$ & $1+\frac{1}{4^s}+\frac{1}{16^s}+\frac{1}{64^s}+\frac{1}{256^s}+\frac{1}{1024^s}+\frac{1}{4096^s}+\frac{3}{6561^s}+\frac{1}{16384^s}+\frac{3}{26244^s}+\frac{1}{65536^s}+\frac{3}{104976^s}+\frac{1}{262144^s}+\cdots$
\\[3pt] $60$ & $1+\frac{3}{256^s}+\frac{3}{625^s}+\frac{3}{6561^s}+\frac{5}{65536^s}+\frac{9}{160000^s}+\frac{5}{390625^s}+\frac{9}{1679616^s}+\frac{9}{4100625^s}+\frac{7}{16777216^s}+\frac{15}{40960000^s}+\cdots$
\\[3pt] $84$ & $1+\frac{3}{2401^s}+\frac{3}{4096^s}+\frac{3}{531441^s}+\frac{5}{5764801^s}+\frac{9}{9834496^s}+\frac{5}{16777216^s}+\frac{9}{1275989841^s}+\frac{9}{2176782336^s}+\cdots$
\\[3pt] \hline
\end{tabular}
}
\end{table}

\begin{table}
\caption{\label{tabK} Complete characterisation of $K$ for all $n$ in Eq. \ref{eq:cls1}.
${\ell}\lrcorner{P_n}$ denotes that an $\ell$-colouring yields $K=T_{(q)}\rtimes P_n$. Other $\ell$-colourings not listed here yield $K=T_{(q)}$. For certain cases with $^*$, see Theorem \ref{thmK}. }
\begin{tabular}{cc|l} \hline
$\phi(n)$ & $n$ & $\ell\overset{\phantom{*}}{\lrcorner}P_n$ \\
 \hline
2  & 3  & $3\lrcorner{D_3}$ \quad ${4}\overset{\phantom{*}}{\lrcorner}{C_2}$\\
   & 4  & $2\lrcorner{D_4}$ \quad ${4}\overset{\phantom{*}}{\lrcorner}{D_2}$ \\
 \hline
4  & 5  & $5\lrcorner{D_5}$ \quad ${16}\overset{\phantom{*}}{\lrcorner}{C_2}$ \\
   & 8  & $2\lrcorner{D_8}$ \quad $4\lrcorner{D_4}$\quad ${8}\lrcorner{C_2}$ \quad ${16}\overset{\phantom{*}}{\lrcorner}{C_2}$ \\
   & 12 & $4\lrcorner{C_4}$ \quad $9\lrcorner{C_3}$ \quad ${16}\overset{\phantom{*}}{\lrcorner}{C_2}$ \\
 \hline
6  & 7  & $7\lrcorner{D_7}$\quad $8\lrcorner{C_2}$ \quad ${64}\overset{*}{\lrcorner}{C_2}$\\
   & 9  & $3\lrcorner{D_9}$ \quad $9\lrcorner{C_3}$\quad ${27}\lrcorner{C_3}$ \quad ${64}\overset{\phantom{*}}{\lrcorner}{C_2}$ \\
 \hline
8  & 15 & $16\lrcorner{C_2}$ \quad $25\lrcorner{C_5}$\quad ${81}\lrcorner{C_3}$ \quad ${256}\overset{*}{\lrcorner}{C_2}$ \\
   & 16 & $2\lrcorner{D_{16}}$ \quad $4\lrcorner{D_8}$\quad ${8}\lrcorner{C_4}$ \quad ${16}\lrcorner{C_4}$ \quad ${{32}}\lrcorner{C_2}$\quad ${{64}}\lrcorner{C_2}$\quad ${{128}}\lrcorner{C_2}$\quad ${256}\overset{\phantom{*}}{\lrcorner}{C_2}$\\
   & 20 & $5\lrcorner{C_5}$\quad ${16}\lrcorner{C_4}$ \quad $25\overset{*}{\lrcorner}{C_5}$ \quad ${256}\lrcorner{C_2}$ \\
   & 24 &  $4\lrcorner{C_8}$\quad$9\lrcorner{C_3}$\quad ${16}\lrcorner{C_4}$ \quad $81\overset{*}{\lrcorner}{C_3}$ \quad ${256}\lrcorner{C_2}$\\
 \hline
10 & 11 & $11\lrcorner{D_{11}}$ \quad ${1024}\overset{\phantom{*}}{\lrcorner}{C_2}$ \\
 \hline
12 & 13 & $13\lrcorner{D_{13}}$ \quad ${4096}\overset{\phantom{*}}{\lrcorner}{C_2}$ \\
   & 21 & $7\lrcorner{C_7}$\quad$49\overset{*}{\lrcorner}{C_7}$\quad ${64}\lrcorner{C_2}$ \quad $729\lrcorner{C_3}$ \quad ${4096}\overset{*}{\lrcorner}{C_2}$ \\
   & 28 & $8\lrcorner{C_4}$\quad ${49}\lrcorner{C_7}$ \quad $64\overset{*}{\lrcorner}{C_4}$\quad ${512}\overset{*}{\lrcorner}{C_2}$ \quad ${4096}\overset{*}{\lrcorner}{C_2}$  \\
   & 36 & $9\lrcorner{C_9}$\quad ${64}\lrcorner{C_4}$ \quad $729\lrcorner{C_3}$ \quad ${4096}\overset{\phantom{*}}{\lrcorner}{C_2}$ \\
 \hline
16 & 17 & $17\lrcorner{D_{17}}$ \quad ${256}\lrcorner{C_2}$\quad ${65536}\overset{*}{\lrcorner}{C_2}$ \\
   & 32 & ${2}\lrcorner{D_{32}}$ \quad ${4}\lrcorner{D_{16}}$ \quad
   ${8,16}\lrcorner{C_{8}}$ \quad ${{32}}\lrcorner{C_4}$\quad ${{64}}\lrcorner{C_4}$\quad ${{128}}\lrcorner{C_4}$\quad ${256}\overset{\phantom{*}}{\lrcorner}{C_4}$  \\ &&
   ${2^i}\overset{\phantom{*}}{\lrcorner}{C_2}$ ($i=9,10,11,12,13,14,15,16$) \\
   & 40 & ${16}\lrcorner{C_{8}}$ \quad ${25}\lrcorner{C_{5}}$ \quad
   ${256}\lrcorner{C_{4}}$ \quad ${625}\overset{*}{\lrcorner}{C_5}$ \quad
   ${4096}\lrcorner{C_2}$ \quad ${65536}\lrcorner{C_2}$  \\
   & 48 &${4}\lrcorner{C_{16}}$ \quad ${16}\lrcorner{C_{8}}$ \quad
   ${64}\lrcorner{C_{4}}$ \quad ${81}\lrcorner{C_3}$ \quad
   ${256}\lrcorner{C_4}$\quad
   ${6561}\overset{*}{\lrcorner}{C_3}$ \\ &&
   ${1024}\lrcorner{C_2}$ \quad ${4096}\lrcorner{C_2}$ \quad ${16384}\lrcorner{C_2}$ \quad ${65536}\overset{\phantom{*}}{\lrcorner}{C_2}$ \\
   & 60 & ${16}\lrcorner{C_{4}}$ \quad ${25}\lrcorner{C_{5}}$ \quad
   ${81}\lrcorner{C_{3}}$ \quad ${256}\overset{*}{\lrcorner}{C_4}$ \quad
   ${625}\overset{*}{\lrcorner}{C_5}$\quad ${4096}\overset{*}{\lrcorner}{C_2}$ \\ &&
   ${6561}\overset{*}{\lrcorner}{C_3}$\quad
   ${65536}\overset{*}{\lrcorner}{C_2}$ \\
 \hline
18 & 19 & $19\lrcorner{D_{19}}$ \quad ${262144}\overset{\phantom{*}}{\lrcorner}{C_2}$ \\
   & 27 & ${3}\lrcorner{D_{27}}$ \quad ${9}\lrcorner{C_{9}}$ \quad
   ${27}\lrcorner{C_{9}}$ \quad ${3^i}\lrcorner{C_3}$ ($i=4,5,6,7,8,9$) \quad
   ${262144}\overset{\phantom{*}}{\lrcorner}{C_2}$ \\
 \hline
20 & 25 & ${5}\lrcorner{D_{25}}$ \quad ${5^i}\lrcorner{C_5}$ ($i=2,3,4,5$) \quad
${1048576}\overset{\phantom{*}}{\lrcorner}{C_2}$\\
   & 33 & ${3}\lrcorner{D_{27}}$ \quad ${9}\lrcorner{C_{9}}$ \quad
   ${27}\lrcorner{C_{9}}$ \quad ${3^i}\lrcorner{C_3}$ ($i=4,5,6,7,8,9$) \quad
   ${1048576}\overset{\phantom{*}}{\lrcorner}{C_2}$ \\
   & 44 & ${121}\lrcorner{C_{11}}$ \quad ${1024}\lrcorner{C_4}$ \quad
   ${1048576}\overset{\phantom{*}}{\lrcorner}{C_2}$ \\
 \hline
24 & 35 & ${2401}\lrcorner{C_{7}}$ \quad ${4096}\lrcorner{C_{2}}$ \quad
   ${15625}\lrcorner{C_{5}}$ \quad ${16777216}\overset{*}{\lrcorner}{C_2}$ \\
   & 45 & ${81}\lrcorner{C_{9}}$ \quad ${4096}\lrcorner{C_{2}}$\quad${6561}\lrcorner{C_{3}}$ \quad
   ${15625}\lrcorner{C_{5}}$\quad${531441}\lrcorner{C_{3}}$ \quad ${16777216}\overset{*}{\lrcorner}{C_2}$ \\
   & 84 & ${49}\lrcorner{C_{7}}$\quad ${64}\lrcorner{C_{4}}$ \quad
   ${729}\lrcorner{C_{3}}$\quad${2401}\overset{*}{\lrcorner}{C_{7}}$\quad
   ${4096}\overset{*}{\lrcorner}{C_{4}}$ \\ &&${531441}\overset{*}{\lrcorner}{C_{3}}$ \quad ${262144}\overset{*}{\lrcorner}{C_2}$ \quad ${16777216}\overset{*}{\lrcorner}{C_2}$ \\
 \hline
\end{tabular}
\end{table}


\begin{thebibliography}{ZZZ9}

\bibitem{b1}
M.\ Baake: Combinatorial aspects of colour symmetries, {\em J.\ Phys.\ A: Math.\ Gen.} 30 (1997) 2687-98.

\bibitem{bg}
M.\ Baake \& U.\ Grimm: Bravais colourings of planar modules with $N$-fold symmetry, {\em Z.\ Krist.}\ 219 (2004) 72-80.


\bibitem{bugs}
E.P.\ Bugarin, M.L.A. de las Pe\~nas, I.\ Evidente, R.P.\ Felix \& D.\ Frettl\"oh: On color groups of Bravais colorings of planar modules with quasicrystallographic symmetry, {\em Z.\ Krist.} 223 (2008) 785-790.

\bibitem{jj}
E.P.\ Bugarin, M.L.A. de las Pe\~nas \& D.\ Frettl\"oh: Perfect colourings of cyclotomic integers. {\em Submitted for publication.} 
{\tt arXiv:0905.4048v1}

\bibitem{philmag}
E.P.\ Bugarin, M.L.A. de las Pe\~nas \& D.\ Frettl\"oh: Colourings of cyclotomic integers with class number one. {\em Phil. Mag.} 33 (2010) 45-54.

\bibitem{mlp}
M.L.A.\ De Las Pe\~nas \& R. Felix: Color groups associated with square and hexagonal lattices. {\em Z.\ Krist.} 222 (2007) 505-512.



\bibitem{gs}
  B.\ Gr\"unbaum \& G.C.\ Shephard:
  {\em Tilings and patterns},
  Freeman, New York, 1987.





\bibitem{MM}
 J.M. Masley, H.L. Montgomery: Cyclotomic fields with unique factorization,
 	{\em J. Reine. Angew. Math.} 286/287 (1975) 246-256.

\bibitem{schw}
  R.L.E.\ Schwarzenberger:
  Colour symmetry,
  {\em Bull.\ London Math.\ Soc.} 16 (1984) 209-240.




\bibitem{wash}
L.C.\ Washington:
{\em Introduction to cyclotomic fields}, 2nd ed. 
Springer, New York (1997).

\end{thebibliography}
\end{document}